\documentclass[12 pt]{article}

\usepackage{amssymb}
\usepackage{bbm}
\usepackage{amsfonts}
\usepackage{mathrsfs}
\usepackage{graphicx}
\usepackage{amsmath}
\usepackage{amsthm}
\usepackage{xypic}
\usepackage{graphicx}
\usepackage{times}
\usepackage{geometry}
\usepackage{color}

\newtheorem{Th}{\scshape  Theorem}[section]
\newtheorem{Lem}[Th]{\scshape  Lemma}
\newtheorem{Cor}[Th]{\scshape Corollary}

\newtheorem{Prop}[Th]{\scshape Proposition}

\begin{document}

\title{Minimal graphs with eigenvalue multiplicity of $n-d$}
\author{
  Yuanshuai Zhang, \quad Dein Wong \thanks{Corresponding author. \quad E-mail addresses:wongdein@163.com.\quad Supported by the National Natural Science Foundation of China(No.11971474)}, \quad  Wenhao Zhen \\
\small School of Mathematics, China University of Mining and Technology, Xuzhou, China \\
}
\date{}
\maketitle

\noindent {\bf Abstract:}
For a connected graph $G$ with order $n$, let $e(G)$ be the number of its distinct eigenvalues and $d$ be the diameter.
We denote by $m_G(\mu)$ the eigenvalue multiplicity of $\mu$ in $G$.
It is well known that $e(G)\geq d+1$,
which shows $m_G(\mu)\leq n-d$ for any real number $\mu$.
A graph is called $minimal$ if $e(G)= d+1$.
In 2013, Wang (\cite{WD}, Linear Algebra Appl.) characterize all minimal graphs with $m_G(0)=n-d$.
In 2023, Du et al. (\cite{Du}, Linear Algebra Appl.) characterize all the trees for which there is a real symmetric matrix with nullity $n-d$ and $n-d-1$.
In this paper, by applying the star complement theory, we prove that if $G$ is not a path and $m_G(\mu)= n-d$, then $\mu \in \{0,-1\}$.
Furthermore, we completely characterize all minimal graphs with $m_G(-1)=n-d$.

\noindent {\bf Keywords:} eigenvalue multiplicity, star complement theory, diameter

\noindent {\bf AMS Subject Classifications:} 05C50

\section{Introduction}
\quad \quad
All graphs in this paper are simple undirected graphs.
As usual, let $G$ be a connected graph with vertex set $V(G)$ and edge set $E(G)$.
We write $n(G)$ to denote the order of $G$, i.e., the number of vertices in $G$.
The adjacency matrix $ A(G)$  of   $G$ is  an $ n\times n$ square matrix whose $(i,j)$ entry takes $1$ if vertices $i$ and $j$ are adjacent in $G$, and it takes $0$ if otherwise.
The rank of a matrix $A$ is denoted by $rk(A)$
The eigenvalues of $A(G)$ are directly called the eigenvalues of $G$.
For a real number $\mu$, we denote by $m_G(\mu )$ the multiplicity of $\mu $ as an eigenvalue of   $G$, where $m_G(\mu)=0$ implies that $\mu$ is not an eigenvalue of $G$.
In this sense that $rk(A(G)-\mu I)+m_G(\mu)=n$.
If $\mu=0$,  $m_G(\mu )$ is also said to be the nullity of $G$.

For a subset $X$ of $ V(G)$, we denote by $G[X]$ the induced subgraph of $G$ with vertex set $X$  and denote $G[V(G)\backslash X]$ by $G-X$. For simplicity, we replace    $G-\{v\}$ with $G-v$. Let $H$ be an induced graph of $G$, $y\in V(G)\backslash V(H)$, we replace $G[(V(H)\cup \{y\}]$ by $H+y$ for simplicity.
We write  $u\sim v$ if  vertices $u$ and $v$ are adjacent in $G$, and $u\nsim v$ if otherwise.
The set of all neighbors of $v\in V(G)$ in $G$ is denoted by $N_G(v)$, and $N_G[v]=N_G(v)\cup \{v\}$ is the closed neighborhood set of $v$ in $G$.
We call $v$ a pendant vertex if $|N_G(v)|=1$.
For $v\in V(G)$ and a subgraph $H$ of $G$, we write $v\sim H$ to mean that $N_H(v)\neq \emptyset$, and say that $v$ is adjacent to $H$.
By $P_n$ (resp., $C_n$, $K_n$), we  denote a path (resp.,  a cycle, a complete graph) of order $n$.

A path $P_k$ has the form $V(P)=\{v_1,v_2,\ldots, v_k\}$ and $E(P)=\{v_1v_2, v_2v_3, \ldots, v_{k-1}v_k\}$,
where the vertices $v_1,v_2, \ldots, v_k$ are all distinct. We say that $P$ is a path from $v_1$ to $v_k$, or a $(v_1,v_k)-$path.
The number of edges of the path is its length.
The distance $d(x,y)$ in $G$ of two vertices $x,y$ is the length of a shortest $(x,y)-$path in $G$, if no such path exists, we define $d(x,y)$ to be infinite.
The greatest distance between any two vertices in $G$ is the diameter of $G$, denoted by $d(G)$.
Let's write $d(G)$ as $d$ if there's no misunderstanding.

Let $P_{d+1}=v_0v_1\cdots v_{d-1}v_d$ is a path of length $d$ and let $W=\{v_{i_1}, v_{i_2},\ldots,v_{i_k}\}$ be a subset of $V(P_{d+1})$. We attach $k$ single vertices $\{u_{i_1}, u_{i_2},\ldots,u_{i_k}\}$ into $P_{d+1}$ such that $N_{P_{d+1}}(u_j)=\{v_j,v_{j+1}\}$ for $j\in \{i_1,i_2,\ldots,i_k\}$. We call the resulting graph $P_{d+1}\diamond W$.
For example, if $d=7$ and $W=\{v_2,v_3,v_5\}$, then $P_{d+1}\diamond W$ is shown in Fig 1.

  \begin{figure}[h!]
    \centering
   \includegraphics[width=5in]{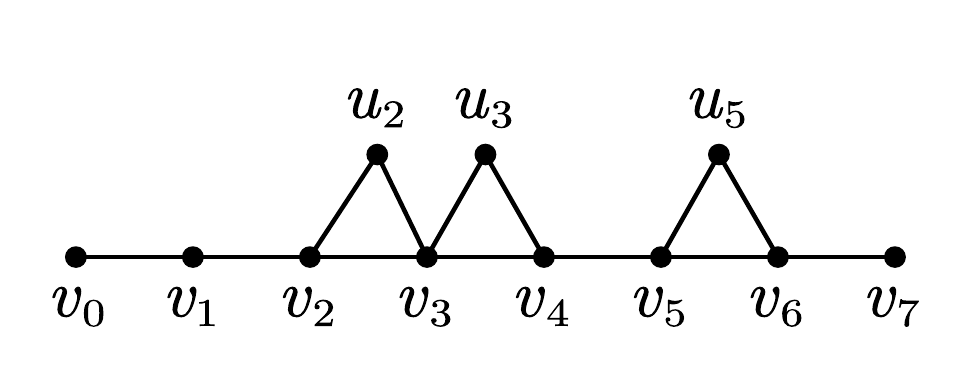}\\
   \caption{$P_8\diamond \{v_2,v_3,v_5\}$}
   \label{1}
   \end{figure}

Let $G$ be a graph and $u,v\in V(G)$. We introduce (see \cite{Pe}) a binary relation $\rho$ in $V(G)$ in the following way: $u\rho v$  if and only if $N_G[u]=N_G[v]$.
The relation $\rho$ is symmetric and transitive.
By this relation the vertex set $V(G)$ is divided to $k$ disjoint subsets $C_1,C_2, \ldots, C_k$, ($1\leq k \leq n, n = |V(G)|$), such
that every graph induced by the set $C_i$ ($i = 1,\ldots, k$) is a complete graph.
The graph $G^c$, obtained from $G$ by identification of all vertices from the same subset $C_i$ ($i = 1,\ldots, k$) is called the ``$C$-canonical graph" of the graph $G$.
If $V(G)$ does not have any pair of vertices which lie in relation $\rho$, then $G^c= G$ and we say that $G$ is a $C$-canonical graph.
For instance, all connected bipartite graphs distinct of $P_2$ are $C$-canonical graphs.

Let $G$ be a graph, we denote by $e(G)$ the number of distinct eigenvalues of $G$. It is known that $e(G)\geq d+1$ if $G$ is a connected graph  (\cite{Brouwer}, Proposition 1.3.3). Further, we have the following proposition.

\begin{Prop}
Suppose $G$ is a simple connected graph on $n$ vertices and $d$ is its diameter. Let $\mu$ be a real number, then
\begin{center}
$m_G(\mu)\leq n-d$.
\end{center}
\end{Prop}

In particular, a graph $G$ is called minimal if $e(G)=d+1$.
Characterizing all minimal graphs is a challenging work for researchers.
In \cite{Bee}, the author provided various minimal graphs by a recursive construction.
By the above proposition, if an eigenvalue $\mu$ of $G$ exists such that $m_G(\mu)= n-d$, then $e(G)=d+1$ and $G$ is a minimal graph.
So, characterizing graphs $G$ with $m_G(\mu)=n-d$ for some real number $\mu$ is a problem worth studying.
In this paper, we obtain the following theorem.

\begin{Th}\label{Th1}
Let $G$ be a connected graph and $\mu$ be a real number. If $G$ is not a path and $m_G(\mu)=n-d$, then $\mu \in \{0,-1\}$.
\end{Th}

Further, we want to characterize all graphs that reach bound $m_G(\mu)=n-d$.
The situation $\mu=0$ has been solved.
In 2013, Wong et al. \cite{WD} characterize graphs with maximum diameter among all connected graph with rank $n$.
In fact, the result in \cite{WD} contains characterization of  all minimal graphs with $m_G(0)=n-d$.
In 2022, Wang \cite{WZ} use the method of $star \  complement$ to characterize all minimal graphs with $m_G(0)=n-d$ again.
In 2023, Du and Fonseca \cite{Du} extend the definition of minimal graphs on adjacency matrices to real symmetric matrices. They characterize all the trees for which there is a real symmetric matrix with nullity $n-d$ and $n-d-1$.
Therefore, we investigate the situation $\mu=-1$.

\begin{Th}\label{Th2}
Let $G$ be a connected graph with order $n$ and diameter $d$. Then $m_G(-1)=n-d$ if and only if $G$ is one of the following forms:\\
(i) $G$ is a complete graph with $n\geq 2$,\\
(ii) $G^c\cong P_5$,\\
(iii) $d\geq 7$, $d\equiv 1(mod\ 3)$ and $G^c$  is isomorphic to $P_{d+1} \diamond W$, where $P_{d+1}=v_0v_1\cdots v_{d-1}v_d$ and $W$ is a subset of the set $U= \{v_3,v_6,\ldots, v_{d-4}\}$.
\end{Th}

In the next section, we give some basic lemmas and give the proof of Theorem \ref{Th1}. In section 3, we characterize all minimal graphs with $m_G(-1)=n-d$.
\section{$m_G(\mu)=n-d$}
\quad \quad Before  giving our main result, we introduce some useful notations and lemmas.

\begin{Prop}\label{RD}
If $H$ is an induced graph of $G$, then $rk(A(H)-\mu I)\leq rk(A(G)-\mu I)$ for any real number $\mu$.
\end{Prop}

\begin{Lem}\label{VD}{\rm (\cite{Brouwer})}\ \
Let v be a vertex of G, then $m_G(\mu)-1 \leq m_{G-v}(\mu)\leq m_G(\mu)+1$.
\end{Lem}

\begin{Lem}\label{PD}{\rm (\cite{Cv2})}\ \
Let $x$ be a pendant vertex of $G$ and $y$ be the unique neighbor of $x$, then $m_G(0)=m_{G-x-y}(0)$.
\end{Lem}

The eigenvalues of $P_n$ are $\{2\cos \frac{i\pi }{n+1} | i=1, 2,\ldots, n\}$, we have the following lemma.

\begin{Lem}\label{PE}
Let $\mu$ be a real number, then $m_{P_n}(\mu)\leq 1$. Further, if $m_{P_{n+1}}(\mu)=1$, then $m_{P_n}(\mu)=0$.
\end{Lem}

Let $G$ be a graph of order $n$ with an eigenvalue $\mu$ and $m_G(\mu )=k$.
A $star\ set$ for $\mu$ in $G$ is a vertex set $X$ such that $|X|=k$, and the subgraph $G-X$ does not have $\mu$ as an eigenvalue.
In this situation, $G-X$ is called a $star\ complement$ for $\mu$ in $G$.
Now we recall some properties of star set.

\begin{Prop}\label{SP}
Let $G$ be a graph, $V(G)=\{1,2,\dots, n\}$, $X\subseteq V(G)$. Let $P$ be the matrix which represents the orthogonal projection of $\mathbb{R} ^n$ onto the eigenspace $\mathcal{E}_{A(G)}(\mu )$ with respect to the standard orthonormal basis $\{e_{1}, e_{2},\dots, e_{n}\}$ of $\mathbb{R} ^n$. Then \\
(i)$X$ is a star set for $\mu$ in $G$ if and only if the vectors $Pe_{i}$, $i\in X$, form a basis for $\mathcal{E}_{A(G)}(\mu )$.\\
(ii)The matrix $P$ of (iii) is a polynomial in $A$ and $\mu Pe_v =\sum_{u\sim v} Pe_u$.
\end{Prop}

\begin{Lem}\label{SS}{\rm (\cite{P3}, Lemma 2.3)}\ \
Let $X$ be a star set for $\mu$ in $G$ and let $U$ be a proper subset of $X$. Then $X\subseteq U$ is a star set for $\mu$ in $G-U$.
\end{Lem}

\begin{Th}\label{DO}{\rm (\cite{Cv}, Proposition 5.1.4)}\ \
Let $X$ be a star set for $\mu$ in $G$, and let $\overline{X} = V(G)\setminus{X}$. If $\mu\neq 0$ , then $\overline{X}$ is a dominating set for $G$.
\end{Th}

Let $G$ be a connected graph, it is easy to prove that the above result also holds for $\mu=0$.

\begin{Lem}\label{DD}
Let $G$ be a connected graph. Let $X$ be a star set for $\mu$ in $G$, and let $\overline{X} = V(G)\setminus{X}$. Then $\overline{X}$ is a dominating set for $G$.
\end{Lem}
\begin{proof}
If $\mu\neq 0$, then the result holds by Theorem \ref{DO}.
If $\mu=0$, let $H=G[\overline{X}]$ be a star complement for $\mu$ in $G$. Suppose that $\overline{X}$ is not a dominating set for $G$, then there exists two vertices $x,y\in X$ such that $x\sim y$, $N_H(x)=\emptyset$ and $N_H(y)\neq \emptyset$. Let $H'=G[\overline{X}\cup\{x,y\}]$, then $x$ is a pendant vertex of $H'$. By Lemma \ref{PD}, $m_H(0)=m_{H'}(0)=0$. Further, $r(H')=|\overline{X}|+2-m_{H'}(0)=|\overline{X}|-m_{H}(0)+2=r(H)+2=r(G)+2$, a contradiction to Lemma \ref{RD}.
\end{proof}

Now we consider the graphs with $m_G(\mu)=n-d$. If $G$ is a path ,then the following lemma is obvious.

\begin{Lem}
Let $P_n$ be a path with order $n$. Then $m_{P_n}(\mu)=n-d(P_n)$ if and only if $\mu \in \{2\cos \frac{i\pi }{n+1} | i=1, 2,\ldots, n\}$.
\end{Lem}

Now we assume $G$ is not a path and give the proof of Theorem \ref{Th1}.

\vskip 2mm\noindent\textbf{Proof of Theorem \ref{Th1}}
 \vskip 2mm
\begin{proof}
Let $G$ be a graph with order $n$ and diameter $d$.
Suppose that $G$ is not a path and $m_G(\mu)=n-d$, then we have $n\geq d+2$ and $d\geq 1$.
If $d=1$, then $G\cong K_n$ is a completed graph. Since $n\geq d+2$, then $m_G(\mu)=n-d$ if and only if $\mu=-1$.
Now we assume that $P_{d+1}=v_0v_1\cdots v_{d-1}v_d$ is a diametrical path of $G$, where $d\geq 2$.

If $m_{P_{d+1}}(\mu)=0$, then by Lemma \ref{VD},$ m_G(\mu) \leq m_{P_{d+1}}(\mu)+n-(d+1)=n-d-1$, a contradiction. Thus, $m_{P_{d+1}}(\mu)=1$.
Let $P'= P_{d+1}-v_0$, by Lemma \ref{PE}, $m_{P'}(\mu)=0$.
Besides, $|P'|=d=n-m_G(\mu)$, which means that $P'$ is a star complement for $\mu$ in $G$.
Let $X=V(G)\setminus V(P_{d+1})$, then $X\neq \emptyset$ since $n\geq d+2$. By Lemma \ref{DD}, $x\sim P'$ for any $x\in X$.
Let $N_{P_{d+1}}(x)=\{i_1,\ldots, i_s\}$. It is easy to see that $1\leq s\leq 3$ since ${P_{d+1}}$ is a diametrical path of $G$, otherwise we can find a shorter path from $v_0$ to $v_d$.
Moreover, $N_{P_{d+1}}(x)\in \{\{v_i\}, \{v_i, v_{i+1}\}, \{v_i, v_{i+2}\}, \{v_i, v_{i+1}, v_{i+2}\}\}$ for some $i$.

Let $G'=P_{d+1}+x$ be an induced graph of $G$, then  by Lemma \ref{SS}, $X'=\{v_0, x\}$ is a star set for $\mu$ in $G'$.
Let $P$ be the matrix which represents the orthogonal projection of $\mathbb{R} ^{d+2}$ onto the eigenspace $\mathcal{E}_{A(G')}(\mu )$ with respect to the standard orthonormal basis $\{e_{v_0}, e_{v_1},\dots, e_{v_d},e_x\}$ of $\mathbb{R} ^{d+2}$.
By (i) of Proposition \ref{SP}, $Pe_{v_0}$ and $Pe_x$ are linearly independent.
There are two cases to discuss here.\\
\textbf{Case 1:} $x\sim v_0$. From the discussion above, three cases need to be considered.

\textbf{Case 1.1:} $N_{P_{d+1}}(x)= \{v_0, v_1\}$.

Now we use the important equality in (ii) of Proposition \ref{SP}.
$$
    \left\{
      \begin{aligned}
        \mu Pe_{v_0}&=Pe_{v_1}+Pe_x,\\
        \mu Pe_{x}&=Pe_{v_1}+Pe_{v_0}.
      \end{aligned}
      \right.
      $$

Then $(\mu +1)Pe_{v_0}=(\mu +1) Pe_x$. If $\mu \notin \{0,-1\}$, then $Pe_{v_0}= Pe_x$, which contradicts the fact that $Pe_{v_0}$ and $Pe_x$ are linearly independent. Notice that if $\mu=-1$, this case is possible.

\textbf{Case 1.2:} $N_{P_{d+1}}(x)= \{v_0, v_2\}$.
$$
    \left\{
      \begin{aligned}
        \mu Pe_{v_1}&=Pe_{v_0}+Pe_{v_2},\\
        \mu Pe_{x}&=Pe_{v_0}+Pe_{v_2}.
      \end{aligned}
      \right.
      $$

If $\mu \notin \{0,-1\}$, then $Pe_{v_1}=Pe_x$ and $\mu Pe_{v_0}=Pe_{v_1}+Pe_x=2Pe_x$, a contradiction. Notice that if $\mu=-1$, this case is also impossible.

\textbf{Case 1.3:} $N_{P_{d+1}}(x)= \{v_0, v_1, v_2\}$.
$$
    \left\{
      \begin{aligned}
        \mu Pe_{v_1}&=Pe_{v_0}+Pe_{v_2}+Pe_x,\\
        \mu Pe_{x}&=Pe_{v_0}+Pe_{v_2}+Pe_{v_1}.
      \end{aligned}
      \right.
      $$

Then $(\mu +1)Pe_{v_1}=(\mu +1) Pe_x$. If $\mu \notin \{0,-1\}$, then $Pe_{v_1}=Pe_x$ and $\mu Pe_{v_0}=Pe_{v_1}+Pe_x=2Pe_x$, a contradiction. Notice that if $\mu=-1$, this case is possible.
\\
\textbf{Case 2:} $x\nsim v_0$. From the discussion above, four cases need to be considered.

\textbf{Case 2.1:} $N_{P_{d+1}}(x)= \{v_i\}$, $1\leq i \leq d$.

By the equalities $\mu Pe_{v_0}= Pe_{v_1}$ and $\mu Pe_{v_j}= Pe_{v_{j-1}}+Pe_{v_{j+1}}$, $1\leq j \leq i-1$, we have $Pe_{v_i} = \textbf{0}$ or $Pe_{v_i} = k Pe_{v_0}$, where \textbf{0} is null vector and $k\neq 0$ is a real number. Besides, $\mu Pe_x= Pe_{v_i}$. If $\mu \notin \{0,-1\}$, then $\mu Pe_x = \textbf{0}$ or $\mu Pe_x = k Pe_{v_0}$, a contradiction. Notice that if $\mu=-1$, this case is also impossible.

\textbf{Case 2.2:} $N_{P_{d+1}}(x)= \{v_i,v_{i+1}\}$, $1\leq i \leq d-1$.
$$
    \left\{
      \begin{aligned}
        \mu Pe_{v_{i}}&=Pe_{v_{i-1}}+Pe_{v_{i+1}}+Pe_x,\\
        \mu Pe_{x}&=Pe_{v_i}+Pe_{v_{i+1}}.
      \end{aligned}
      \right.
      $$

Then $(\mu+1)Pe_x=(\mu+1)Pe_{v_{i}}-Pe_{v_{i-1}}$. Similar to the discussion in Case 2.1, we can assume that $Pe_{v_{i}}=kPe_{v_0}$ and $Pe_{v_{i-1}}=lPe_{v_0}$, where $k,l$ are real numbers. It can be seen that $k,l$ cannot both be 0, otherwise we can deduce $Pe_{v_0}=\textbf{0}$. Suppose that $\mu \notin \{0,-1\}$, then
\begin{itemize}
\item If $k=0$, then $(\mu+1)Pe_x = -Pe_{v_{i-1}}= -lPe_{v_0} \neq \textbf{0}$, a contradiction.
Notice that if $\mu=-1$, this case is also impossible.
\item If $k\neq 0$ and $l=0$, then $(\mu+1)Pe_x=(\mu+1)Pe_{v_{i}}= k(\mu+1)Pe_{v_0}$, a contradiction.
Notice that if $\mu=-1$, this case is possible.
\item If $k\neq 0$ and $l\neq 0$, then $\textbf{0} \neq (\mu+1)Pe_x=(\mu+1)Pe_{v_{i}}-Pe_{v_{i-1}}= (k\mu+k-l)Pe_{v_0}$, a contradiction.
Notice that if $\mu=-1$, this case is also impossible.
\end{itemize}

\textbf{Case 2.3:} $N_{P_{d+1}}(x)= \{v_i,v_{i+2}\}$, $1\leq i \leq d-2$.

$$
    \left\{
      \begin{aligned}
        \mu Pe_{v_{i+1}}&=Pe_{v_i}+Pe_{v_{i+2}},\\
        \mu Pe_x&=Pe_{v_i}+Pe_{v_{i+2}}.
      \end{aligned}
      \right.
      $$

If $\mu \notin \{0,-1\}$, then $Pe_{v_{i+1}}=Pe_x$ and $\mu Pe_{v_i}-Pe_{v_{i-1}}=Pe_x+Pe_{v_{i+1}}=2Pe_x \neq \textbf{0}$. On the other hand, $\mu Pe_{v_i}-Pe_{v_{i-1}}= kPe_{v_0}$, where $k$ is a real number, a contradiction.
Notice that if $\mu=-1$, this case is also impossible.

\textbf{Case 2.4:} $N_{P_{d+1}}(x)= \{v_i,v_{i+1},v_{i+2}\}$, $1\leq i \leq d-2$.
$$
    \left\{
      \begin{aligned}
        \mu Pe_{v_{i+1}}&=Pe_{v_i}+Pe_{v_{i+2}}+Pe_x,\\
        \mu Pe_x&=Pe_{v_i}+Pe_{v_{i+2}}+Pe_{v_{i+1}}.
      \end{aligned}
      \right.
      $$

If $\mu \notin \{0,-1\}$, then $Pe_{v_{i+1}}=Pe_x$ and $\mu Pe_{v_i}-Pe_{v_{i-1}}=Pe_x+Pe_{v_{i+1}}=2Pe_x \neq \textbf{0}$. On the other hand, $\mu Pe_{v_i}-Pe_{v_{i-1}}= kPe_{v_0}$,  where $k$ is a real number, a contradiction.
Notice that if $\mu=-1$, this case is possible.

To sum up,  if $G$ is not a path and $\mu \notin \{0,-1\}$, then $m_G(\mu)\leq n-d-1$, as desired.
\end{proof}

\section{Graphs with $m_G(-1)=n-d$}

\quad \quad In \cite{WD}, all graphs with $m_G(0)=n-d$ are characterized. In this section, we will characterize all graphs with $m_G(-1)=n-d$. We start this section with some useful lemmas.

\begin{Lem}\label{NE}
Let $H$ be an induced subgraph of a connected graph $G$ for which $rank(A(H)+I)\geq rank(A(G)+I)-1$, and $v\in V(G)\backslash V(H)$.
If $v\sim h$ and $N_H(v)=N_H[h]$ for a vertex $h\in V(H)$, then $N_G[v]=N_G[h]$.
\end{Lem}

\begin{proof}
Suppose there exists a vertex $x\in V(G)\backslash V(H)$ such that $x\in N_G[v]$ and $x\notin N_G[h]$. Let $G_1=H+v+x$ be an induced subgraph of $G$.
Arrange the vertices of $G_1$ such that

\begin{center}
    $A(G_1)+I=
    \bordermatrix{
      & x        & v        & h                     &                     \cr
    x & 1        & 1        & 0                     &\beta^\top                    \cr
    v & 1        & 1        & 1                     &\alpha^\top                     \cr
    h & 0        & 1        & 1                     &\alpha ^\top              \cr
      & \beta    & \alpha   & \alpha                &A(H-h)+I          \cr
    },
    $
  \end{center}
where $\alpha, \beta$ are an $|H|-1$ dimensional vector.
Let

\begin{center}
    $Q=
    \bordermatrix{
      & x        & v        & h                     &                     \cr
    x & 1        & 0        & 0                     &0                    \cr
    v & -1       & 1        & 0                     &-\beta^\top                      \cr
    h & 1        & -1       & 1                     &\beta^\top                \cr
      & 0        & 0        & 0                     &I_k         \cr
    },
    $
  \end{center}
  where $k=|H|-1$. Then

\begin{center}
    $Q\top (A(G_1)+I)Q=
    \bordermatrix{
      & x        & v        & h                     &                     \cr
    x & 0        & 1        & 0                     &0                   \cr
    v & 1        & 0        & 0                     &0                    \cr
    h & 0        & 0       & 1                     &\alpha ^\top              \cr
      & 0         & 0       & \alpha                &A(H-h)+I          \cr
    }
    =\bordermatrix{
      & x        & v                          &                     \cr
    x & 0        & 1                             &0                   \cr
    v & 1        & 0                            &0                    \cr
      & 0         & 0                     &A(H)+I          \cr
    }.
    $
  \end{center}
Thus, $rank(A(G_1)+I)=rank(A(H)+I)+2\geq rank(A(G)+I)+1$, which contradicts Lemma \ref{RD}.
So we have $N_G[v]\subseteq N_G[h]$. Let $H_1= H+v-h$ be an induced subgraph of $G$.
Then $H_1$ and $H$ are isomorphic since $v\sim h$ and $N_H(v)=N_H[h]$. Furthermore, $rank(A(H_1)+I)\geq rank(A(G)+I)-1$ and $N_{H_1}(h)=N_{H_1}[v]$.
Now viewing $H_1$ as $H$ and viewing $v$ as $h$, then by the discussion above we conclude that $N_G[h]\subseteq N_G[v]$.
Thus the result follows.
\end{proof}

\begin{Cor}\label{NEE}
Let $H$ be an induced subgraph of a connected graph $G$ for which $rk(A(H)+I)\geq rk(A(G)+I)-1$, and $u, v\in V(G)\backslash V(H)$.
If $u\sim v$ and $N_H(u)=N_H(v)$ for a vertex $h\in V(H)$, then $N_G[u]=N_G[v]$.
\end{Cor}

\begin{proof}
Let $H'=H+u$, then $rk(A(H')+I)\geq rk(A(H)+I)\geq rk(A(G)+I)-1$ and $N_{H'}(v)=N_{H'}[u]$. By Lemma \ref{NE}, $N_G[u]=N_G[v]$.
\end{proof}

\begin{Lem}\label{GPD}{\rm (\cite{Pe}, Theorem 5)}\ \
Let $G$ be a graph. If two vertices $u,v$ satisfy  $u\rho v$, then $m_G(-1)=m_{G-u}(-1)+1=m_{G-v}(-1)+1$.
\end{Lem}

\begin{Lem}\label{Gc}
Let $G^c$ be the $C$-canonical graph of a graph $G$. If $d(G^c)\geq 1$, then $d(G^c)=d(G)$.
\end{Lem}

\begin{proof}
Recall that the vertex set $V(G)$ is divided into $k$ disjoint subsets $C_1,C_2, \ldots, C_k$, ($1\leq k \leq n, n = |V(G)|$), such
that every graph induced by the set $C_i$ ($i = 1,\ldots, k$) is a complete graph.
We assume $V(G^c)=\{c_1,c_2,\ldots c_k\}$, where $c_i$ corresponds to $C_i$ in $G$.
Let $u\in C_i$, $v\in C_j$ be two distinct vertices. If $i=j$, then $d_G(u,v)=1$.
If $i\neq j$, we prove $d_G(u,v)=d_{G^c}(c_i,c_j)$.
Suppose $d_G(u,v)= a$ and $d_{G^c}(c_i,c_j)=b$.
Let $c_ic_{l_1}\ldots c_{l_{b-1}}c_j$ be a path of length $b$ from $c_i$ to $c_j$ in $G^c$.
Then there exists $x_{l_k}\in C_{l_k}$($k=1,\ldots,b-1$) such that $ux_{l_1}\ldots x_{l_{b-1}}v$ is a path of length $b$ from $u$ to $v$ in $G$.
We have $d_G(u,v)=a\leq b$.
Conversely, given a path of length $a$ from $u$ to $v$ in $G$, we can get a corresponding walk of length $a$ from $c_i$ to $c_j$ in $G^c$, which says that $d_{G^c}(v_i,v_j)=b\leq a$. Hence, $d_G(u,v)=d_{G^c}(c_i,c_j)$.
Further, we know that $d(G)= max\{M,N\}$, where

$$
    \left\{
      \begin{aligned}
        M&=max\{d_{G}(u,v)| u \in C_i, v\in C_j , 1\leq i\neq j\leq k\},\\
        N&=max\{d_{G}(u,v)| u,v\in C_i, 1\leq i\leq k \}.
      \end{aligned}
      \right.
      $$
One easily finds that $M=d(G^c)$ and $N=1$. Since $d(G^c)\geq 1$, then $d(G^c)=d(G)$.
\end{proof}

\begin{Lem}\label{Gcc}
Let $G^c$ be the $C$-canonical graph of a graph $G$ with $d(G)\geq 1$. Then $m_G(-1)=n(G)-d(G)$ if and only if $m_{G^c}(-1)=n(G^c)-d(G^c)$.
\end{Lem}
\begin{proof}
If $d(G^c)\geq 1$, then $d(G^c)=d(G)$ by Lemma \ref{Gc}.
Meanwhile, by Lemma \ref{GPD}, $m_G(-1)=m_{G^c}(-1)+(n(G)-n(G^c))$. If $m_G(-1)=n(G)-d(G)$, then $m_{G^c}(-1)+(n(G)-n(G^c))=n(G)-d(G^c)$, i.e.,  $m_{G^c}(-1)=n(G^c)-d(G^c)$, and vice versa.
\end{proof}

By the discussion above, a graph $G$ satisfies $m_G(-1)=n(G)-d(G)$ if and only if the $C$-canonical graph $G^c$ of $G$ also satisfy $m_{G^c}(-1)=n(G^c)-d(G^c)$. In the following, we just have to think about the case where $G$ is a $C$-canonical graph.

\vskip 2mm\noindent\textbf{Proof of Theorem \ref{Th2}}
 \vskip 2mm
 \begin{proof}
 \textbf{`Necessity'} \ \
Let $G$ be a connected order $n\geq 2$, diameter $d$.
If $G\cong K_n$ is a complete graph, then $m_G(-1)=n-1=n-d$, (i) holds.
Let's assume $d\geq 1$ and $m_G(-1)=n-d$.
By Lemma \ref{Gcc}, $m_G(-1)=n(G)-d(G)$ if and only if $m_{G^c}(-1)=n(G^c)-d(G^c)$.
In the following, we just need to study the structure of $G^c$.
For the sake of simplicity, we assume $G=G^c$ is a $C$-canonical graph.

One can easily see that $d\neq 1$ for all $C$-canonical graphs.
Let $d\geq 2$, we assume that $P_{d+1}=v_0v_1\cdots v_{d-1}v_d$ is a diametrical path of $G$.
Let $P'= v_1\cdots v_{d-1}v_d$ and $x\in V(G)\backslash V(P)$.
Now we give some claims.

\textbf{Claim 1:} $d\equiv 1(mod\ 3)$.\\
Firstly, by the proof of Theorem \ref{Th1}, $m_{P_{d+1}}(-1)=1$ and $P'$ is a star complement for $-1$ in $G$.
Since the eigenvalues of $P_{d+1}$ are $\{2\cos \frac{i\pi }{d+2} | i=1, 2,\ldots, d+1\}$, then $m_{P_{d+1}}(-1)=1$ if and only if $d\equiv 1(mod\ 3)$.

\textbf{Claim 2:} $N_{P_{d+1}}(x)\in \{\{v_i,v_{i+1}\}, \{v_i,v_{i+1}, v_{i+2}\}\}$ for some $i$.\\
Since $m_{G}(-1)=n-d$, by the proof of Theorem \ref{Th1}, there are four possible cases: case 1.1, case 1.3, case 2.2 and case 2.4.
The claim is obvious.

\textbf{Claim 3:} $N_{P_{d+1}}(x)=\{v_i,v_{i+1}\}$ for some $i$.

If $N_{P_{d+1}}(x)=\{v_i,v_{i+1}, v_{i+2}\}$, $0\leq i\leq d-2$, then $N_{P_{d+1}}(x)=N_{P_{d+1}}[v_{i+1}]$.
Since $rk(A(G)+I)=n-m_{G}(-1)=d$ and $rk(A(P_{d+1})+I)=d$, then $rk(A(P_{d+1})+I)\geq rk(A({G})+I)-1$.
By Lemma \ref{NE}, $N_{G}[x]=N_{G}[v_{i+1}]$, i.e., $x\rho v_{i+1}$ in ${G}$, which contradicts ${G}$ is a connected $C$-canonical graph.
Thus, $N_{P_{d+1}}(x)=\{v_i,v_{i+1}\}$ for some $i$.

\textbf{Claim 4:} If $N_{P_{d+1}}(x)=\{v_i,v_{i+1}\}$, then $1\leq i\leq d-2$ and $i\equiv 0 (mod \ 3)$.

If $i=0$, then $N_{P_{d+1}}(x)=N_{P_{d+1}}[v_0]$. By Lemma \ref{NE}, we can conclude $x\rho v_0$ in $G$, a contradiction.
If $i=d-1$, then $N_{P_{d+1}}(x)=N_{P_{d+1}}[v_{d}]$ and $x\rho v_{d}$ in $G$, a contradiction. Thus, $1\leq i\leq d-2$.
Furthermore, we just need to consider the case 2.2 in the proof of Theorem \ref{Th1}.
Follow the definition $\{G', X, P, e_{v_i}\}$ in the proof of Theorem \ref{Th1}.
Since $m_G(-1)=n-d$, then by case 2.2, $Pe_{v_{i}}\neq \mathbf{0}$ and $Pe_{v_{i-1}}=\mathbf{0}$.
Meanwhile, it is easy to check that

$$
    \left\{
      \begin{aligned}
        Pe_{v_{j}}&=Pe_{v_0},\ \  &j\equiv 0(mod\ 3), 0\leq j \leq i,\\
        Pe_{v_{j}}&=-Pe_{v_0},   &j\equiv 1(mod\ 3),0\leq j \leq i,\\
        Pe_{v_{j}}&=\mathbf{0},  &j\equiv 2(mod\ 3),0\leq j \leq i.
      \end{aligned}
      \right.
      $$
Consequently, $i-1\equiv 2(mod\ 3)$ and $i\equiv 0(mod\ 3)$, as desired.

In the following, let $y\in V(G)\backslash V(P)$ be a different vertex from $x$. Then we have Claim 5 and Claim 6.

\textbf{Claim 5:} If $N_{P_{d+1}}(x)=\{v_i,v_{i+1}\}$ and $N_{P_{d+1}}(y)=\{v_j,v_{j+1}\}$, then $i\neq j$.

Suppose $i=j$. If $x\sim y$, then by Corollary \ref{NEE}, $x\rho y$ in $G$, a contradiction.
If$x\nsim y$, then by simple elementary transformations we have $rk(A(P+x+y)+I)=rk(B)$, where
\begin{center}
    $B=
    \bordermatrix{
      & x        & y                             &                     \cr
    x & 2        & 0                             &0                   \cr
    y & 0       & \frac{1}{2}                            &\alpha^\top                    \cr
      & 0         & \alpha                      &A(P_{d+1})+I          \cr
    }.
    $
  \end{center}
Since $rk(B)\geq rk(A(P_{d+1})+I )+1=d+1$, then $rk(A(P+x+y)+I)> rk(A(G)+I)$, which contradicts Lemma \ref{RD}.

\textbf{Claim 6:} If$N_{P_{d+1}}(x)=\{v_i,v_{i+1}\}$ and $N_{P_{d+1}}(y)=\{v_j,v_{j+1}\}$, then $x\nsim y$.

By Claim 4 and Claim 5, $j-i\geq 3$. Without loss of generality, we assume $j>i$.
 If $x\sim y$, then there exists a shorter walk $v_0v_1\ldots v_ixyv_{j+1}\ldots v_d$ from $v_0$ to $v_d$, a contradiction.

Binding Claim 1 to Claim 5, if $d\geq 2$, then $G$ is isomorphic to $P_5$ or $P_{d+1} \diamond W$, where $d\equiv 1(mod\ 3)$ and $W$ is a subset of the set $U= \{v_3,v_6,\ldots, v_{d-4}\}$.
We have done.

 \textbf{`Sufficiency'}\ \  \
 Obviously, $m_{G}(-1)=n-d$ when $G$ is a complete graph or $G^c\cong P_5$.
 Suppose $d\geq 7$, $d\equiv 1(mod\ 3)$ and $G^c$  is isomorphic to $P_{d+1} \diamond W$, where $P_{d+1}=v_0v_1\cdots v_{d-1}v_d$ and $W$ is a subset of the set $U= \{v_3,v_6,\ldots, v_{d-4}\}$.
 Let $S=\{v_j| j\equiv 2 (mod \ 3), 1\leq j \leq d\}$ and $G'=G^c-S$.
 By the definition of $P_{d+1} \diamond W$, one can see that $|W|=n(G^c)-(d(G^c)+1)$ and $G'= aC_3 \cup bK_2$, where $a=|W|$ and $b=|S|+1-|W|$. Thus, $m_{G'}(-1)=2a+b$.
 Meanwhile, $m_{G'}(-1)\leq m_{G^c}(-1)+|S|$ by Lemma \ref{VD}. Combining the above formulas, we have $m_{G^c}(-1)\geq |W|+1=n(G^c)-d(G^c)$. On the other hand, $m_{G^c}(-1)\leq n(G^c)-d(G^c)$, we have done.

\end{proof}

\end{document}